\documentclass{article}
\usepackage[%
journal=???,    
lang=british,   
]{ems-journal}

\usepackage{esint,dsfont}
\usepackage[backgroundcolor=olive, linecolor=olive, textsize=scriptsize]{todonotes}

\def\R{{\mathbb R}}
\def\N{\mathbb{N}}
\def\C{\mathbb{C}}
\def\Z{\mathbb{Z}}
\def\D{\mathbb{D}}

\def\F{\mathbb{F}}

\def\ii{\mathrm{i}}

\theoremstyle{plain}
\newtheorem{thm}{Theorem}[section]
\newtheorem{lem}[thm]{Lemma}
\newtheorem{prop}[thm]{Proposition}
\newtheorem{cor}[thm]{Corollary}

\theoremstyle{definition}

\newtheorem{rmk}[thm]{Remark}
\newtheorem{ques}[thm]{Question}

\begin{document}

\title{Some remarks on $M_d$-multipliers and approximation properties}
\titlemark{$M_d$-multipliers and approximation properties}



\emsauthor{1}{
	\givenname{Ignacio}
	\surname{Vergara}
	\mrid{1126379}
	\orcid{0000-0001-7144-4272}}{I.~Vergara}

\Emsaffil{1}{
	\department{Departamento de Matem\'atica y Ciencia de la Computaci\'on}
	\organisation{Universidad de Santiago de Chile}
	\address{Las Sophoras 173}
	\zip{9170020}
	\city{Estaci\'on Central}
	\country{Chile}
	\affemail{ign.vergara.s@gmail.com}}

\classification[43A07, 46L07, 46B28, 20E08, 22E30]{43A22}


\begin{abstract}
We prove an extension property for $M_d$-multipliers from a subgroup to the ambient group, showing that $M_{d+1}(G)$ is strictly contained in $M_d(G)$ whenever $G$ contains a free subgroup. Another consequence of this result is the stability of the $M_d$-approximation property under group extensions. We also show that Baumslag--Solitar groups are $M_d$-weakly amenable with $\boldsymbol\Lambda(\operatorname{BS}(m,n),d)=1$ for all $d\geq 2$. Finally, we show that, for simple Lie groups with finite centre, $M_d$-weak amenability is equivalent to weak amenability, and we provide some estimates on the constants $\boldsymbol\Lambda(G,d)$.
\end{abstract}

\maketitle



\section{Introduction}
This paper is concerned with $M_d$-multipliers of locally compact groups, and various notions of approximation properties associated to them. This class of functions was first introduced by Pisier \cite{Pis} for discrete groups, as a tool to study the Dixmier similarity problem. The definition was later extended to all locally compact groups by Battseren \cite{Bat,Bat2}, who also coined the term \emph{$M_d$-multiplier}.

Let $G$ be a locally compact group, and let $C_b(G)$ denote the algebra of bounded, continuous, complex-valued functions on $G$. For Banach spaces $E,F$, let $\mathbf{B}(E,F)$ denote the space of bounded linear operators from $E$ to $F$. Let $d\geq 2$ be an integer. We say that $\varphi\in C_b(G)$ is an $M_d$-multiplier of $G$ if there are bounded maps $\xi_i:G\to\mathbf{B}(\mathcal{H}_i,\mathcal{H}_{i-1})$ ($i=1,\ldots,d$), where $\mathcal{H}_i$ is a Hilbert space, $\mathcal{H}_0=\mathcal{H}_d=\C$, and
\begin{align}\label{phi_xi}
\varphi(t_1\cdots t_d)=\xi_1(t_1)\cdots \xi_d(t_d)1
\end{align}
for all $t_1,\ldots,t_d\in G$. We let $M_d(G)$ denote the space of $M_d$-multipliers of $G$, and we endow it with the norm
\begin{align*}
\|\varphi\|_{M_d(G)}=\inf\left\{\sup_{t_1\in G}\|\xi_1(t_1)\|\cdots\sup_{t_d\in G}\|\xi_d(t_d)\|\right\},
\end{align*}
where the infimum is taken over all decompositions as in \eqref{phi_xi}. With this norm, $M_d(G)$ becomes a Banach algebra for pointwise operations. Observe that $M_{d+1}(G)\subseteq M_d(G)$ for all $d\geq 2$.

For $d=2$, $M_2(G)$ is the algebra of Herz--Schur multipliers, which is at the heart of the definition of weak amenability \cite{CowHaa,Ver3}, and other approximation properties generalising it, such as the AP \cite{HaaKra} and the weak Haagerup property \cite{Knu}. It turns out that similar approximation properties can be defined analogously, using the algebra $M_d(G)$ instead.

The $M_d$-approximation property ($M_d$-AP) was introduced in \cite{Ver} as a strengthening of the AP of Haagerup and Kraus \cite{HaaKra}, with the goal of giving a partial answer to the Dixmier problem. In order to define it, we need to view $M_d(G)$ as a dual Banach space. The general definition for locally compact groups that we present here is due to Battseren \cite{Bat,Bat2}. Let $L^1(G)$ denote the $L^1$ space on $G$, endowed with a left Haar measure. We define the space $X_d(G)$ as the completion of $L^1(G)$ for the norm
\begin{align*}
\|g\|_{X_d(G)}=\sup\left\{\left|\int_G\varphi(t)g(t)\,dt\right|\ \mid\ \varphi\in M_d(G),\ \|\varphi\|_{M_d(G)}\leq 1\right\}.
\end{align*}
Then $X_d(G)^*=M_d(G)$ for the duality
\begin{align*}
\langle\varphi,g\rangle=\int_G\varphi(t)g(t)\,dt
\end{align*}
for all $\varphi\in M_d(G)$, $g\in L^1(G)$; see \cite[Theorem 0.3]{Bat2}. Let us mention that, when $G$ is discrete, $X_d(G)$ may also be defined as a quotient of the $n$-fold Haagerup tensor product $\ell^1(G)\otimes_h\cdots\otimes_h\ell^1(G)$; see \cite[\S 3]{Pis}. The locally compact case is more subtle; see \cite{Bat2} for details.

Let $C_c(G)$ be the subalgebra of $C_b(G)$ given by all continuous, compactly supported functions on $G$. We say that $G$ has the $M_d$-AP if the constant function 1 belongs to the $\sigma(M_d(G),X_d(G))$-closure of $C_c(G)$ in $M_d(G)$. For every $d\geq 2$, we have
\begin{align*}
M_{d+1}\text{-AP}\implies M_d\text{-AP}
\end{align*}
because the inclusion $M_{d+1}(G)\hookrightarrow M_d(G)$ is weak*-weak*-continuous. Moreover, $M_2$-AP is exactly the AP of Haagerup and Kraus \cite{HaaKra}. It is not known whether any of the implications above is an equivalence.

The main motivation behind the definitions of $M_d$-multipliers and $M_d$-AP is the study of the Dixmier problem. A group $G$ is said to be unitarisable if every uniformly bounded representation of $G$ on a Hilbert space is similar to a unitary representation. This property is satisfied by $\Z$ \cite{Szo}, and, more generally, by every amenable group \cite{Day,Dix,NakTak}. The Dixmier problem asks whether the converse is also true: is every unitarisable group amenable? This question remains open, but some partial answers have been given. The following result was proved in \cite{Ver}.

\begin{thm}[{\cite[Theorem 1.2]{Ver}}]\label{Thm_unit+MdAP}
Let $G$ be a discrete group. If $G$ is unitarisable and satisfies $M_d$-AP for all $d\geq 2$, then it is amenable.
\end{thm}

In light of this result, it becomes relevant to determine how large the class of groups satisfying $M_d$-AP is. In particular, the following question remains open.

\begin{ques}
Is $M_2$-AP equivalent to $M_d$-AP for all $d\geq 3$?
\end{ques}

Let us mention that $M_2$-AP (AP) is a very weak property. When it was introduced in \cite{HaaKra}, the only known examples of groups failing to satisfy this property were non-exact groups; see \cite[\S 12.4]{BroOza}. After considerable work, the list was expanded in order to include higher rank algebraic groups and their lattices \cite{HaKndL,HaadLa,HaadLa2,LafdlS,Lia}, and $\tilde{A}_2$-lattices \cite{LeSaWi}. To the author's knowledge, no more examples have been found.

In \cite{Ver}, several examples of groups satisfying $M_d$-AP were given, including all groups acting properly on finite-dimensional $\operatorname{CAT}(0)$ cube complexes; see \cite[Theorem 1.3]{Ver}. Moreover, it was shown in \cite[Lemma 4.3]{Ver} that $M_d$-AP is stable under extensions, with the additional hypothesis that the normal subgroup appearing in the exact sequence is amenable. Our first result asserts that this is true in general.

\begin{thm}\label{Thm_MdAP_ext}
Let $G$ be a discrete group, $\Gamma$ a normal subgroup of $G$, and $d\geq 2$. If both $\Gamma$ and $G/\Gamma$ satisfy $M_d$-AP, then so does $G$.
\end{thm}

In particular, we get the following corollary.

\begin{cor}\label{Cor_MdAP_prod}
For every $d\geq 2$, the $M_d$-AP for discrete groups is stable under direct products, semidirect products, and free products.
\end{cor}

The proof of Theorem \ref{Thm_MdAP_ext} relies on the fact that elements of $M_d(\Gamma)$ may be viewed as elements of $M_d(G)$ by extending them by $0$; see Lemma \ref{Lem_ext}. As a byproduct of this extension property, we obtain the following result, generalising \cite[Theorem 5.1]{Pis}.

\begin{prop}\label{Prop_free_sub}
Let $G$ be a discrete group containing a nonabelian free subgroup. Then, for every $d\geq 2$,
\begin{align*}
M_{d+1}(G)\subsetneq M_d(G).
\end{align*}
\end{prop}

\begin{rmk}
It would be very interesting to determine whether Proposition \ref{Prop_free_sub} can be generalised to the setting of random embeddings; see \cite[\S 3]{Pis3} for a precise definition. The main motivation for studying this question is that, as a consequence of the celebrated Gaboriau--Lyons theorem \cite{GabLyo}, an infinite group $G$ is amenable if and only if the free group $\F_2$ cannot be realised as a ``random subgroup'' of $G$; see \cite[Corollary 12]{Pis3}. An analogous result to Proposition \ref{Prop_free_sub} in this setting would completely settle the Dixmier problem. Indeed, by \cite[Theorem 2.9]{Pis}, for every unitarisable group $G$, there exists $d_0\geq 2$ such that $M_d(G)=M_{d_0}(G)$ for all $d\geq d_0$.
\end{rmk}

Continuing our search for examples, we turn to the notion of $M_d$-weak amenability. We say that a locally compact group $G$ is $M_d$-weakly amenable if there is $C\geq 1$ such that the constant function $1$ is in the $\sigma(M_d(G),X_d(G))$-closure of the set
\begin{align*}
\left\{\varphi\in C_c(G)\ \mid\ \|\varphi\|_{M_d(G)}\leq C\right\}
\end{align*}
in $M_d(G)$. We define $\boldsymbol\Lambda(G,d)$ as the infimum of all $C\geq 1$ such that the condition above holds. For $d=2$, this property is exactly weak amenability, as defined by Cowling and Haagerup \cite{CowHaa}, and $\boldsymbol\Lambda(G,2)$ is the Cowling--Haagerup constant $\boldsymbol\Lambda(G)$. It can be seen from the definition that every $M_d$-weakly amenable group satisfies $M_d$-AP. Moreover, since the inclusion $M_{d+1}(G)\hookrightarrow M_d(G)$ is contractive, we always have
\begin{align*}
	\boldsymbol\Lambda(G,d)\leq\boldsymbol\Lambda(G,d+1).
\end{align*}
For convenience, when $G$ is not $M_d$-weakly amenable, we simply set $\boldsymbol\Lambda(G,d)=\infty$.

The first concrete examples that we analyse are Baumslag--Solitar groups, which are defined by the following presentation. For $m,n\in\Z\setminus\{0\}$,
\begin{align*}
	\operatorname{BS}(m,n)=\big\langle a,b\ \mid\ a^n=ba^mb^{-1}\big\rangle.
\end{align*}
It was shown in \cite{GalJan} that $\operatorname{BS}(m,n)$ can be realised as a closed subgroup of a locally compact group of the form $\big(\Z\ltimes\R\big)\times\operatorname{Aut}(T)$, where $\operatorname{Aut}(T)$ is the automorphism group of a locally finite tree. As a consequence, $\operatorname{BS}(m,n)$ has the Haagerup property. The same argument shows that $\operatorname{BS}(m,n)$ is weakly amenable with $\boldsymbol\Lambda(\operatorname{BS}(m,n))=1$; see \cite{CorVal}. Here, we strengthen this fact as follows.

\begin{thm}\label{Thm_BS}
Let $d\geq 2$, and $m,n\in\Z\setminus\{0\}$. Then $\operatorname{BS}(m,n)$ is $M_d$-weakly amenable with $\boldsymbol\Lambda(\operatorname{BS}(m,n),d)=1$.
\end{thm}

In order to prove Theorem \ref{Thm_BS}, we need to show that $\boldsymbol\Lambda(\operatorname{Aut}(T),d)=1$, and that the constant $\boldsymbol\Lambda(\,\cdot\,,d)$ is submultiplicative; see Corollary \ref{Cor_Aut(T)} and Lemma \ref{Lem_Mdwa_prod}. Then we can use the embedding $\operatorname{BS}(m,n)\hookrightarrow\big(\Z\ltimes\R\big)\times\operatorname{Aut}(T)$ given by \cite{GalJan}.

Lastly, we focus on Lie groups. For a simple Lie group $G$, weak amenability is characterised by its real rank; see Section \ref{S_Lie_gr} for the definition of $\operatorname{rank}_\R G$. More precisely, $G$ is weakly amenable if and only if $\operatorname{rank}_\R G$ is $0$ or $1$; see e.g. \cite[\S 5]{Ver3}. Moreover, the exact value of the Cowling--Haagerup constant $\boldsymbol\Lambda(G)$ depends only on the local isomorphism class of $G$. In \cite{CowHaa}, Cowling and Haagerup proved that $\boldsymbol\Lambda(\operatorname{Sp}(n,1))=2n-1$ and $\boldsymbol\Lambda(\operatorname{F}_{4,-20})=21$, providing the first examples of groups for which $\boldsymbol\Lambda(G)$ is strictly between $1$ and $\infty$. A very important consequence of this result is the fact that two lattices $\Gamma<\operatorname{Sp}(n,1)$, $\Lambda<\operatorname{Sp}(m,1)$ cannot have isomorphic von Neumann algebras if $n\neq m$. For $M_d$-weak amenability, we prove the following.

\begin{thm}\label{Thm_Lie_gr}
Let $G$ be a simple Lie group with finite centre, and let $d\geq 2$. Then $G$ is $M_d$-weakly amenable if and only if it has real rank $0$ or $1$. Moreover,
\begin{align*}
&\boldsymbol\Lambda(G,d)=1  & &\text{if }\ \operatorname{rank}_\R G=0,\\
&\boldsymbol\Lambda(G,d)=1  & &\text{if }\ G\approx\operatorname{SO}(n,1),\ n\geq 2,\\
&\boldsymbol\Lambda(G,d)=1  & &\text{if }\ G\approx\operatorname{SU}(n,1),\ n\geq 2,\\
2n-1 \leq\ &\boldsymbol\Lambda(G,d) \leq (2n-1)^d  & &\text{if }\ G\approx\operatorname{Sp}(n,1),\ n\geq 2,\\
21 \leq\ &\boldsymbol\Lambda(G,d) \leq (21)^d & &\text{if }\ G\approx\operatorname{F}_{4,-20}.
\end{align*}
\end{thm}

It was shown in \cite[Theorem 0.7]{Bat2} that, if $\Gamma$ is a lattice in $G$, then $\boldsymbol\Lambda(\Gamma,d)=\boldsymbol\Lambda(G,d)$ for all $d\geq 2$. Therefore, Theorem \ref{Thm_Lie_gr} also applies to lattices. Moreover, for discrete groups, the constants $\boldsymbol\Lambda(\Gamma,d)$ are invariant under von Neumann equivalence; see \cite[Theorem 1.1]{Bat}. This implies that $\boldsymbol\Lambda(\Gamma,d)=\boldsymbol\Lambda(\Lambda,d)$ whenever $\Gamma$ and $\Lambda$ have isomorphic von Neumann algebras. This gives a new tool for distinguishing between group von Neumann algebras; however, it is still not clear whether $M_d$-weak amenability is really different to ($M_2$-)weak amenability. More precisely, we do not know if it is possible to have
\begin{align*}
	\boldsymbol\Lambda(G,d)<\boldsymbol\Lambda(G,d+1)
\end{align*}
for some $d\geq 2$.

Let us also mention that lattices in rank $1$ simple Lie groups are hyperbolic. A natural question is whether the result above can be extended to all hyperbolic groups.

\begin{ques}
	Are hyperbolic groups $M_d$-weakly amenable for all $d\geq 2$?
\end{ques}

For $d=2$, this question has a positive answer; see \cite{Oza}.

\begin{rmk}\label{Rmk_G=KS}
	The main tool in the proof of Theorem \ref{Thm_Lie_gr} is a family of approximate identities constructed in \cite{Ver2}, which in turn are given by a construction of uniformly bounded representations from \cite{Doo}. One could alternatively try to adapt the arguments in \cite{deCHaa} and \cite{CowHaa} with the goal of calculating the exact values of $\boldsymbol\Lambda(G,d)$. This was indeed our first attempt. Everything seems to work with minor modifications, except for \cite[Proposition 1.6(ii)]{CowHaa}, which relates coefficients of unitary representations on $S$ to elements of $M_2(G)$ when $G=KS$ for $K$ compact and $S$ amenable. It is not clear whether this result can be extended to $M_d(G)$.
\end{rmk}

This paper is organised as follows. In Section \ref{S_ext_mult}, we prove an extension property for $M_d$-multipliers, together with Proposition \ref{Prop_free_sub}. In Section \ref{S_MdAP_ext}, we focus on the stability of $M_d$-AP and prove Theorem \ref{Thm_MdAP_ext}. Section \ref{S_Mdwa} is devoted to $M_d$-weak amenability and various general results that will be needed later. In Section \ref{S_BS_gr}, we discuss Baumslag--Solitar groups and the proof of Theorem \ref{Thm_BS}. Finally, in Section \ref{S_Lie_gr}, we focus on Lie groups and the proof of Theorem \ref{Thm_Lie_gr}.

\section{Extending multipliers from a subgroup}\label{S_ext_mult}
The goal of this section is to show that, when $G$ is a discrete group and $\Gamma$ is a subgroup of $G$, elements of $M_d(\Gamma)$ may be viewed as elements of $M_d(G)$ by extending them to $G\setminus\Gamma$ by $0$. This will be achieved through the use of a cocycle $\alpha:G\times G/\Gamma\to\Gamma$.

Let $q:G\to G/\Gamma$ be the quotient map. We say that $\sigma:G/\Gamma\to G$ is a lifting if $q\circ\sigma=\operatorname{id}_{G/\Gamma}$. We will also impose the condition $\sigma(q(e))=e$, where $e$ denotes the identity element of $G$. Fix such a lifting, and observe that
\begin{align*}
G=\bigsqcup_{x\in G/\Gamma}\sigma(x)\Gamma.
\end{align*}
Hence, for all $s\in G$ and $x\in G/\Gamma$, there is a unique element $\alpha(s,x)\in\Gamma$ such that
\begin{align*}
s\sigma(x)=\sigma(q(s\sigma(x)))\alpha(s,x).
\end{align*}
Observe that
\begin{align*}
\sigma(q(s\sigma(x)))=\sigma(sq(\sigma(x)))=\sigma(sx),
\end{align*}
where $sx$ is given by the action by left multiplication of $G$ on $G/\Gamma$. Therefore we can define $\alpha:G\times G/\Gamma\to\Gamma$ by
\begin{align}\label{def_cocy}
\alpha(s,x)=\sigma(sx)^{-1}s\sigma(x)
\end{align}
for all $s\in G$ and $x\in G/\Gamma$. It readily follows that $\alpha$ satisfies the cocycle identity:
\begin{align}\label{cocy_id}
\alpha(st,x)=\alpha(s,tx)\alpha(t,x)
\end{align}
for all $s,t\in G$ and $x\in G/\Gamma$. This cocycle will allow us to prove the extension property that we are after. Let $\C[G]$ denote the group algebra of $G$. For $f\in\C[G]$, we denote by $f|_\Gamma$ the restriction of $f$ to $\Gamma$.

\begin{lem}\label{Lem_ext}
Let $G$ be a discrete group, $\Gamma$ a subgroup of $G$, and $d\geq 2$. The linear map
\begin{align*}
f\in\C[G]\mapsto f|_\Gamma\in\C[\Gamma]
\end{align*} extends to a bounded map $\Upsilon:X_d(G)\mapsto X_d(\Gamma)$ of norm $1$. Its dual map $\Upsilon^*:M_d(\Gamma)\to M_d(G)$ is given by
\begin{align}\label{Upsilon*}
\Upsilon^*(\varphi)(s)=\begin{cases}
\varphi(s), & s\in\Gamma,\\
0,& \text{otherwise},
\end{cases}
\end{align}
for all $\varphi\in M_d(\Gamma)$.
\end{lem}
\begin{proof}
We will first show that the formula \eqref{Upsilon*} gives a well defined contraction from $M_d(\Gamma)$ to $M_d(G)$, and then we will prove that it is the dual map of $\Upsilon$. Let $\varphi\in M_d(\Gamma)$ be given by
\begin{align*}
\varphi(s_1\cdots s_d)=\xi_1(s_1)\cdots\xi_d(s_d)
\end{align*}
for all $s_1,\ldots,s_d\in\Gamma$, where the maps $\xi_i:\Gamma\to\mathbf{B}(\mathcal{H}_i,\mathcal{H}_{i-1})$ are as in \eqref{phi_xi}. Let us define
\begin{align*}
	\tilde{\mathcal{H}}_0=\tilde{\mathcal{H}}_d=\C,
\end{align*}
and
\begin{align*}
\tilde{\mathcal{H}}_i=\ell^2(G/\Gamma)\otimes\mathcal{H}_i
\end{align*}
for all $i=1,\ldots,d-1$. Fix a lifting $\sigma:G/\Gamma\to G$ and a cocycle $\alpha:G\times G/\Gamma\to\Gamma$ as in \eqref{def_cocy}, and define $\tilde{\xi}_d:G\to\mathbf{B}(\tilde{\mathcal{H}}_d,\tilde{\mathcal{H}}_{d-1})$ by
\begin{align*}
\tilde{\xi}_d(s)1=\delta_{q(s)}\otimes\xi_d(\alpha(s,q(e)))1
\end{align*}
for all $s\in G$. We see that
\begin{align*}
\|\tilde{\xi}_d(s)1\|_{\tilde{\mathcal{H}}_{d-1}}^2=\|\xi_d(\alpha(s,q(e)))1\|_{\mathcal{H}_{d-1}}^2,
\end{align*}
which shows that
\begin{align*}
\sup_{s\in G}\|\tilde{\xi}_d(s)\|\leq\sup_{t\in\Gamma}\|\xi_d(t)\|.
\end{align*}
If $d\geq 3$, we define $\tilde{\xi}_i:G\to\mathbf{B}(\tilde{\mathcal{H}}_i,\tilde{\mathcal{H}}_{i-1})$ ($i=2,\ldots,d-1$) by
\begin{align*}
\tilde{\xi}_i(s)(\delta_x\otimes v)=\delta_{sx}\otimes\xi_i(\alpha(s,x))v
\end{align*}
for all $s\in G$, $x\in G/\Gamma$, $v\in\mathcal{H}_i$. Hence, for every choice of pairwise distinct points $x_1,\ldots,x_n\in G/\Gamma$, and every $v_1,\ldots,v_n\in\mathcal{H}_i$,
\begin{align*}
\bigg\|\tilde{\xi}_i(s)\bigg(\sum_{j=1}^n\delta_{x_j}\otimes v_j\bigg)\bigg\|^2
&=\sum_{j=1}^n\|\xi_i(\alpha(s,x_j))v_j\|^2\\
&\leq\bigg(\sup_{t\in\Gamma}\|\xi_i(t)\|\bigg)^2\sum_{j=1}^n \|v_j\|^2\\
&=\bigg(\sup_{t\in\Gamma}\|\xi_i(t)\|\bigg)^2\bigg\|\sum_{j=1}^n\delta_{x_j}\otimes v_j\bigg\|^2,
\end{align*}
which shows that
\begin{align*}
\sup_{s\in G}\|\tilde{\xi}_i(s)\|\leq\sup_{t\in\Gamma}\|\xi_i(t)\|.
\end{align*}
Finally, we define $\tilde{\xi}_1:G\to\mathbf{B}(\tilde{\mathcal{H}}_1,\tilde{\mathcal{H}}_0)$ by
\begin{align*}
\tilde{\xi}_1(s)(\delta_x\otimes v)=\langle\delta_{sx},\delta_{q(e)}\rangle\xi_1(\alpha(s,x))v
\end{align*}
for all $s\in G$, $x\in G/\Gamma$, $v\in\mathcal{H}_1$. Again, we have
\begin{align*}
\sup_{s\in G}\|\tilde{\xi}_1(s)\|\leq\sup_{t\in\Gamma}\|\xi_1(t)\|.
\end{align*}
Now, for every $s_1,\ldots,s_d\in G$,
\begin{align*}
\tilde{\xi}_1(s_1)&\cdots\tilde{\xi}_d(s_d)1 \\
&=\tilde{\xi}_1(s_1)\cdots\tilde{\xi}_{d-1}(s_{d-1})(\delta_{q(s_d)}\otimes\xi_d(\alpha(s_d,q(e)))1)\\
&=\tilde{\xi}_1(s_1)\cdots\tilde{\xi}_{d-2}(s_{d-2})(\delta_{q(s_{d-1}s_d)}\otimes\xi_{d-1}(\alpha(s_{d-1},q(s_d)))\xi_d(\alpha(s_d,q(e)))1)\\
&\ \ \vdots\\
&= \tilde{\xi}_1(s_1)(\delta_{q(s_2\cdots s_d)}\otimes\xi_{2}(\alpha(s_2,q(s_3\cdots s_d)))\cdots\xi_d(\alpha(s_d,q(e)))1)\\
&= \langle\delta_{q(s_1\cdots s_d)},\delta_{q(e)}\rangle\xi_1(\alpha(s_1,q(s_2\cdots s_d)))\cdots\xi_d(\alpha(s_d,q(e)))1\\
&= \langle\delta_{q(s_1\cdots s_d)},\delta_{q(e)}\rangle \varphi(\alpha(s_1,q(s_2\cdots s_d))\cdots\alpha(s_d,q(e))).
\end{align*}
By the identity \eqref{cocy_id}, this equals
\begin{align*}
\langle\delta_{q(s_1\cdots s_d)},\delta_{q(e)}\rangle\varphi(\alpha(s_1\cdots s_d, q(e))).
\end{align*}
On the other hand, for every $s\in\Gamma$, we have $\alpha(s,q(e))=s$. This shows that
\begin{align*}
\tilde{\xi}_1(s_1)\cdots\tilde{\xi}_d(s_d)1=\tilde{\varphi}(s_1\cdots s_d),
\end{align*}
where
\begin{align*}
\tilde{\varphi}(s)=\begin{cases}
\varphi(s), & \text{if } s\in\Gamma,\\
0, & \text{otherwise}.
\end{cases}
\end{align*}
We conclude that $\tilde{\varphi}$ belongs to $M_d(G)$. Moreover, by the previous computations,
\begin{align*}
\|\tilde{\varphi}\|_{M_d(G)}\leq \|\varphi\|_{M_d(\Gamma)}.
\end{align*}
Now recall that $\Upsilon:\C[G]\to\C[\Gamma]$ is given by $\Upsilon(f)=f|_\Gamma$. The estimate above, together with the identity
\begin{align*}
\langle\varphi,\Upsilon(f)\rangle=\langle\tilde{\varphi},f\rangle
\end{align*}
shows that $\Upsilon$ extends to a bounded map $X_d(G)\to X_d(\Gamma)$ of norm $1$ whose dual map $\Upsilon^*:M_d(\Gamma)\to M_d(G)$ is given by
\begin{align*}
\Upsilon^*(\varphi)&=\tilde{\varphi}.\qedhere
\end{align*}
\end{proof}

We can now prove that $M_{d+1}(G)\subsetneq M_d(G)$ when $G$ contains a free subgroup.

\begin{proof}[Proof of Proposition \ref{Prop_free_sub}]
Since $G$ contains a nonabelian free subgroup, it contains a copy of $\F_\infty$; see the proof of \cite[Corollary D.5.3]{CecCoo}. Let $d\geq 2$, and $\varphi\in M_d(\F_\infty)\setminus M_{d+1}(\F_\infty)$, which exists by \cite[Theorem 5.1]{Pis}. By Lemma \ref{Lem_ext}, the function $\tilde{\varphi}:G\to\C$ given by
\begin{align*}
\tilde{\varphi}(s)=\begin{cases}
\varphi(s), & s\in\F_\infty,\\
0,& \text{otherwise},
\end{cases}
\end{align*}
belongs to $M_d(G)$, and
\begin{align*}
\|\tilde{\varphi}\|_{M_d(G)}\leq \|\varphi\|_{M_d(\F_\infty)}.
\end{align*}
On the other hand, $\tilde{\varphi}$ does not belong to $M_{d+1}(G)$. Indeed, if this were the case, then the restriction of $\tilde{\varphi}$ to $\F_\infty$ would be an element of $M_{d+1}(\F_\infty)$; see \cite[\S 2]{Pis}. As this restriction is exactly $\varphi$, this is not possible. We conclude that
\begin{align*}
\tilde{\varphi}&\in M_d(G)\setminus M_{d+1}(G).\qedhere
\end{align*}
\end{proof}

\begin{rmk}
Pisier showed in \cite[Theorem 2.9]{Pis} that, if $G$ is unitarisable, then there is $d_0\geq 2$ such that $M_d(G)=M_{d_0}(G)$ for all $d\geq d_0$. Thus Proposition \ref{Prop_free_sub} gives a new proof of the fact that a group containing a nonabelian free subgroup is not unitarisable; see \cite[Theorem 2.7]{Pis2}.
\end{rmk}

\section{$M_d$-AP and group extensions}\label{S_MdAP_ext}
In this section, we prove Theorem \ref{Thm_MdAP_ext}. As was mentioned in the introduction, this was proved in \cite[Lemma 4.3]{Ver} in the particular case when the subgroup $\Gamma$ is amenable. Lemma \ref{Lem_ext} is the ingredient that was missing for the argument to work in full generality. Hence we can now simply repeat the proof of \cite[Lemma 4.3]{Ver} in our more general setting.

\begin{proof}[Proof of Theorem \ref{Thm_MdAP_ext}]
	We fix $G$, $\Gamma$, and $d\geq 2$ such that both $\Gamma$ and $G/\Gamma$ satisfy $M_d$-AP. For each $f\in\C[G]$, let $\Phi_f:\C[G]\to \C[G]$ be the convolution map
\begin{align*}
\Phi_f(g)=f\ast g.
\end{align*}
Observing that
\begin{align*}
\|\Phi_f(g)\|_{X_d(G)} \leq \sum_{s\in G} |f(s)|\|\delta_s\ast g\|_{X_d(G)} \leq \|f\|_1\|g\|_{X_d(G)},
\end{align*}
we see that $\Phi_f$ extends to a bounded map $\Phi_f:X_d(G)\to X_d(G)$ of norm at most $\|f\|_1$. Now let $\Upsilon:X_d(G)\to X_d(\Gamma)$ be the map given by Lemma \ref{Lem_ext}. Defining $\Psi_f=\Upsilon\circ\Phi_f$, we get a bounded map from $X_d(G)$ to $X_d(\Gamma)$ such that, for all $g\in\C[G]$,
\begin{align*}
\Psi_f(g)=\left.\left(f\ast g\right)\right|_\Gamma
\end{align*}
Then the adjoint map $\Psi_f^*:M_d(\Gamma)\to M_d(G)$ is weak*-weak*-continuous. A simple calculation shows that, for all $\varphi\in M_d(\Gamma)$,
\begin{align*}
\Psi_f^*(\varphi)=\check{f}\ast\Upsilon^*(\varphi),
\end{align*}
where $\check{f}(t)=f(t^{-1})$. Now, since $\Gamma$ satisfies $M_d$-AP, there is a net $(\varphi_i)$ in $\C[\Gamma]$ converging to $1$ in $\sigma(M_d(\Gamma),X_d(\Gamma))$. Thus, $\Psi_f^*(\varphi_i)$ converges to $\check{f}\ast\mathds{1}_\Gamma$ in $\sigma(M_d(G),X_d(G))$, where $\mathds{1}_\Gamma$ is the indicator function of $\Gamma$ in $G$. Therefore
\begin{align}\label{incl_w*cl}
\left\{ f\ast\mathds{1}_\Gamma\ :\ f\in\C[G]\right\} \subseteq \overline{\C[G]}^{\sigma(M_d(G),X_d(G))}.
\end{align}
The rest of the proof consists in showing that the constant function $1$ is in the\linebreak $\sigma(M_d(G),X_d(G))$-closure of the left hand side of \eqref{incl_w*cl}, which is done in the exact same way as in the proof of \cite[Lemma 4.3]{Ver} since it relies only on the fact that $G/\Gamma$ satisfies $M_d$-AP. We give the main ideas here, and refer the reader to \cite{Ver} for details. Let $q:G\to G/\Gamma$ be the quotient map. The map $\Theta:M_d(G/\Gamma)\to M_d(G)$, defined by $\Theta(\psi)=\psi\circ q$, is\linebreak weak*-weak*-continuous. Taking a net $(\psi_i)$ in $\C[G/\Gamma]$ converging to $1$ in\linebreak $\sigma(M_d(G/\Gamma),X_d(G/\Gamma))$, we find $f_i$ in $\C[G]$ such that
\begin{align*}
\Theta(\psi_i)=\psi_i\circ q=f_i\ast\mathds{1}_\Gamma.
\end{align*}
Hence $f_i\ast\mathds{1}_\Gamma$ converges to $1$ in $\sigma(M_d(G),X_d(G))$.
\end{proof}

Now we prove the stability of $M_d$-AP under (semi-)direct products and free products.

\begin{proof}[Proof of Corollary \ref{Cor_MdAP_prod}]
Fix $d\geq 2$. Let us consider first the case of semidirect products. Let $G_1, G_2$ be two discrete groups satisfying $M_d$-AP, and such that $G_1$ acts on $G_2$ by automorphisms. This action allows us to define the semidirect product $G_1\ltimes G_2$; see \cite[\S 5.4]{DruKap} for details. We have the following exact sequence:
\begin{align*}
1 \to G_1 \to G_1\ltimes G_2 \to G_2 \to 1.
\end{align*}
Then, by Theorem \ref{Thm_MdAP_ext}, $G_1\ltimes G_2$ satisfies $M_d$-AP. Since a direct product is a particular case of a semidirect product, where the defining action is trivial, we conclude that $M_d$-AP is also stable under direct products. Finally, for a free product, we have the following exact sequence:
\begin{align*}
1 \to F \to G_1\ast G_2 \to G_1\times G_2 \to 1,
\end{align*}
where $F$ is a free group; see e.g. \cite[\S 4.5]{Yan}. By the previous discussion, $G_1\times G_2$ satisfies $M_d$-AP. Moreover, by \cite[Theorem 1.3]{Ver}, $F$ satisfies $M_d$-AP too. Therefore, by Theorem \ref{Thm_MdAP_ext}, so does $G_1\ast G_2$.
\end{proof}

\section{$M_d$-weak amenability}\label{S_Mdwa}

Now we turn to $M_d$-weak amenability. Recall that a locally compact group $G$ is $M_d$-weakly amenable ($d\geq 2$) if there is $C\geq 1$ such that the constant function $1$ is in the $\sigma(M_d(G),X_d(G))$-closure of the set
\begin{align*}
	\left\{\varphi\in C_c(G)\ \mid\ \|\varphi\|_{M_d(G)}\leq C\right\}.
\end{align*}
The constant $\boldsymbol\Lambda(G,d)$ is defined as the infimum of all $C\geq 1$ such that the condition above holds. This property may be reinterpreted as the existence of an approximate identity in the Fourier algebra $A(G)$ that is bounded for the norm of $M_d(G)$. In order to clearly state this characterisation, we need to review some facts about representations.

\subsection{Matrix coefficients of representations}
Let $G$ be a locally compact group, and let $\pi:G\to\mathbf{B}(\mathcal{H})$ be a linear representation, where $\mathcal{H}$ is a Hilbert space. We say that $\pi$ is uniformly bounded if
\begin{align*}
	|\pi|=\sup_{s\in G}\|\pi(s)\| < \infty.
\end{align*}
We will only consider representations that are continuous for the strong operator topology, meaning that the map
\begin{align*}
	s\in G \longmapsto \pi(s)\xi \in\mathcal{H}
\end{align*}
is continuous for every $\xi\in\mathcal{H}$. We say that $\varphi:G\to\C$ is a coefficient of $\pi$ if there are $\xi,\eta\in\mathcal{H}$ such that, for every $s\in G$,
\begin{align}\label{coef_rep}
	\varphi(s)=\langle\pi(s)\xi,\eta\rangle.
\end{align}
Following \cite{Pis}, for every $\theta\geq 1$, we let $B_\theta(G)$ denote the space of all coefficients of representations $\pi$ of $G$ with $|\pi|\leq\theta$. We endow this space with the norm 
\begin{align*}
	\|\varphi\|_{B_\theta(G)}=\inf\|\xi\| \|\eta\|,
\end{align*}
where the infimum is taken over all decompositions as in \eqref{coef_rep}, with $|\pi|\leq\theta$. As in the case of $M_d(G)$, this is a dual space. Let $\tilde{A}_\theta(G)$ be the completion of $L^1(G)$ for the norm
\begin{align*}
	\|g\|_{\tilde{A}_\theta(G)}=\sup\left\{\left|\int_G\varphi(t)g(t)\,dt\right|\ \mid\ \varphi\in B_\theta(G),\ \|\varphi\|_{ B_\theta(G)}\leq 1\right\}.
\end{align*}
Then $B_\theta(G)$ can be identified with the dual space of $\tilde{A}_\theta(G)$; see \cite[Proposition 2.10]{Ver2}. We will need the following fact.

\begin{lem}\label{Lem_Btheta_Md}
	Let $G$ be a locally compact group, and let $d\geq 2$ be an integer. For every $\theta\geq 1$, the inclusion $B_\theta(G)\hookrightarrow M_d(G)$ is a weak*-weak*-continuous map of norm at most $\theta^d$.
\end{lem}
\begin{proof}
	Let $\varphi\in B_\theta(G)$, and write
	\begin{align*}
		\varphi(s)=\langle\pi(s)\xi,\eta\rangle
	\end{align*}
	as in \eqref{coef_rep}. Then, for all $s_1,\ldots,s_d\in G$,
	\begin{align*}
		\varphi(s_1 \cdots s_d) = \langle\pi(s_1)\cdots\pi(s_d)\xi,\eta\rangle.
	\end{align*}
	This shows that $\varphi$ is an element of $M_d(G)$, and
	\begin{align*}
		\|\varphi\|_{M_d(G)} \leq \theta^d 	\|\varphi\|_{B_\theta(G)}.
	\end{align*}
	Therefore the inclusion $B_\theta(G)\hookrightarrow M_d(G)$ is well defined and has norm at most $\theta^d$. The fact that it is weak*-weak*-continuous follows from observing that this inclusion is the dual map of the identity $L^1(G)\to L^1(G)$, when we endow $L^1(G)$ with the norm of $X_d(G)$ and $\tilde{A}_\theta(G)$ respectively.
\end{proof}

When $\theta=1$, $B_\theta(G)$ is called the Fourier--Stieltjes algebra of $G$, and we denote it by $B(G)$; we refer the reader to \cite{KanLau} for a detailed presentation of $B(G)$. This is the space of coefficients of unitary representations of $G$, and it is a Banach algebra for pointwise operations. One can define a subalgebra of $B(G)$ by looking at a very particular representation. The left regular representation $\lambda:G\to\mathbf{B}(L^2(G))$ is defined by
\begin{align*}
	\lambda(s)f(t)=f(s^{-1}t)
\end{align*}
for all $s,t\in G$, $f\in L^2(G)$. The Fourier algebra $A(G)$ is the subalgebra of $B(G)$ given by all coefficients of $\lambda$. In principle, $A(G)$ is simply a subset of $B(G)$, but it can be shown that it is actually an ideal. Moreover, $A(G)$ can be alternatively defined as the closure of $C_c(G)$ in $B(G)$; see \cite[Proposition 2.3.3]{KanLau}.

The following result is an adaptation of \cite[Lemma 2.2]{Haa} to our setting; see \cite[Proposition 0.5]{Bat2} and \cite[Remark 2.3]{Bat} for more details.

\begin{prop}\label{Prop_Mdwa_ucc}
Let $G$ be a locally compact group, $d\geq 2$ an integer, and $C>1$. The following are equivalent:
\begin{itemize}
\item[(i)] The group $G$ is $M_d$-weakly amenable with $\boldsymbol\Lambda(G,d)<C$.
\item[(ii)] For every compact subset $K\subseteq G$ and every $\varepsilon>0$, there is $\varphi\in A(G)$ such that $\|\varphi\|_{M_d(G)}<C$ and
\begin{align*}
\sup_{x\in K}|\varphi(x)-1| < \varepsilon.
\end{align*}
\item[(iii)] For every compact subset $K\subseteq G$ and every $\varepsilon>0$, there is $\varphi\in C_c(G)$ such that $\|\varphi\|_{M_d(G)}<C$ and
\begin{align*}
\sup_{x\in K}|\varphi(x)-1| < \varepsilon.
\end{align*}
\end{itemize}
\end{prop}

\subsection{Direct products}
Now we show that $M_d$-weak amenability is preserved under direct products. This fact will be crucial for the proof of Theorem \ref{Thm_BS}. We begin with the following observation; see \cite[Corollary 1.8]{deCHaa} for the case $d=2$.

\begin{lem}\label{Lem_prod_M_d}
Let $G,H$ be two locally compact groups, and let $d\geq 2$. Let $\varphi_1\in M_d(G)$ and $\varphi_2\in M_d(H)$, and define $\varphi:G\times H\to\C$ by
\begin{align*}
\varphi(x,y)=\varphi_1(x)\varphi_2(y)
\end{align*}
for all $x\in G$, $y\in H$. Then $\varphi$ belongs to $M_d(G\times H)$, and
\begin{align*}
\|\varphi\|_{M_d(G\times H)}\leq\|\varphi_1\|_{M_d(G)}\|\varphi_2\|_{M_d(H)}.
\end{align*}
\end{lem}
\begin{proof}
First observe that $\varphi$ is continuous because both $\varphi_1$ and $\varphi_2$ are. Now let $C_1>\|\varphi_1\|_{M_d(G)}$ and $C_2>\|\varphi_2\|_{M_d(H)}$. By definition, there are Hilbert spaces $\mathcal{H}_0,\ldots,\mathcal{H}_d$ with $\mathcal{H}_0=\mathcal{H}_d=\C$, and bounded maps $\xi_i:G\to\mathbf{B}(\mathcal{H}_i,\mathcal{H}_{i-1})$ ($i=1,\ldots,d$) such that
\begin{align*}
\varphi_1(x_1\cdots x_d)=\xi_1(x_1)\cdots \xi_d(x_d)1
\end{align*}
for all $x_1,\ldots,x_d\in G$, and
\begin{align*}
\left(\sup_{x_1\in G}\|\xi_1(x_1)\|\right)\cdots\left(\sup_{x_d\in G}\|\xi_d(x_d)\|\right) < C_1.
\end{align*}
Similarly, we find bounded maps $\eta_i:H\to\mathbf{B}(\mathcal{K}_i,\mathcal{K}_{i-1})$ ($i=1,\ldots,d$) such that
\begin{align*}
\varphi_2(y_1\cdots y_d)=\eta_1(y_1)\cdots \eta_d(y_d)1
\end{align*}
for all $ y_1,\ldots,y_d\in H$, and
\begin{align*}
\left(\sup_{y_1\in H}\|\eta_1(y_1)\|\right)\cdots\left(\sup_{y_d\in H}\|\eta_d(y_d)\|\right) < C_2.
\end{align*}
Defining $\psi_i:G\times H\to\mathbf{B}(\mathcal{H}_i\otimes\mathcal{K}_i,\mathcal{H}_{i-1}\otimes\mathcal{K}_{i-1})$ by
\begin{align*}
\psi_i(x_i,y_i)=\xi_i(x_i)\otimes\eta_i(y_i),
\end{align*}
we get, for all $x_1,\ldots,x_d\in G$ and $y_1,\ldots,y_d\in H$,
\begin{align*}
\varphi((x_1,y_1)\cdots(x_1,y_d))&=(\xi_1(x_1)\otimes\eta_1(y_1))\cdots(\xi_d(x_d)\otimes\eta_d(y_d))1\\
&= \psi_1(x_1,y_1)\cdots\psi_d(x_d,y_d)1,
\end{align*}
which shows that $\varphi$ belongs to $M_d(G\times H)$, and
\begin{align*}
\|\varphi\|_{M_d(G\times H)}&\leq \left(\sup_{(x_1,y_1)\in G\times H}\|\psi_1(x_1,y_1)\|\right)\cdots\left(\sup_{(x_d,y_d)\in G\times H}\|\psi_d(x_d,y_d)\|\right)\\
&\leq \left(\sup_{x_1\in G}\|\xi_1(x_1)\|\right)\left(\sup_{y_1\in H}\|\eta_1(y_1)\|\right)
\cdots \left(\sup_{x_d\in G}\|\xi_d(x_d)\|\right)\left(\sup_{y_d\in H}\|\eta_d(y_d)\|\right)\\
&< C_1 C_2.
\end{align*}
Since $C_1>\|\varphi_1\|_{M_d(G)}$ and $C_2>\|\varphi_2\|_{M_d(H)}$ were arbitrary, we conclude that
\begin{align*}
\|\varphi\|_{M_d(G\times H)}&\leq\|\varphi_1\|_{M_d(G)}\|\varphi_2\|_{M_d(H)}.\qedhere
\end{align*}
\end{proof}

With this characterisation, we can prove the following stability result.

\begin{lem}\label{Lem_Mdwa_prod}
Let $G,H$ be two locally compact groups, and let $d\geq 2$ be an integer. Then $G\times H$ is $M_d$-weakly amenable if and only if both $G$ and $H$ are. Moreover, in this case,
\begin{align*}
\boldsymbol\Lambda(G\times H,d)\leq\boldsymbol\Lambda(G,d)\boldsymbol\Lambda(H,d).
\end{align*}
\end{lem}
\begin{proof}
Assume first that $G$ and $H$ are $M_d$-weakly amenable, and let $C_1>\boldsymbol\Lambda(G,d)$, $C_2>\boldsymbol\Lambda(H,d)$. Let $K$ be a compact subset of $G\times H$ and $\varepsilon>0$. Then there are compact subsets $K_1\subseteq G$, $K_2\subseteq H$ such that
\begin{align*}
K\subseteq K_1\times K_2.
\end{align*}
By Proposition \ref{Prop_Mdwa_ucc}, there are $\varphi_1\in C_c(G)$ and $\varphi_2\in C_c(H)$ such that
\begin{align*}
\|\varphi_1\|_{M_d(G)}&<C_1, & \|\varphi_2\|_{M_d(H)}&<C_2,
\end{align*}
and
\begin{align*}
\sup_{x\in K_1}|\varphi_1(x)-1| &<\delta, & \sup_{y\in K_2}|\varphi_2(y)-1| &<\delta,
\end{align*}
with $\delta$ small enough so that $\delta^2+2\delta<\varepsilon$. Now, by Lemma \ref{Lem_prod_M_d}, the function $\varphi:G\times H\to\C$, defined by
\begin{align*}
\varphi(x,y)=\varphi_1(x)\varphi_2(y),
\end{align*}
satisfies
\begin{align*}
\|\varphi\|_{M_d(G\times H)}\leq \|\varphi_1\|_{M_d(G)}\|\varphi_2\|_{M_d(H)} < C_1C_2.
\end{align*}
Moreover, it is compactly supported because both $\varphi_1$ and $\varphi_2$ are. Finally, for every $(x,y)\in K$,
\begin{align*}
|\varphi(x,y)-1| &= |\varphi_1(x)\varphi_2(y)-\varphi_1(x)+\varphi_1(x)-1|\\
&\leq |\varphi_1(x)| |\varphi_2(y)-1| + |\varphi_1(x)-1|\\
&\leq (1+\delta)\delta + \delta\\
&<\varepsilon.
\end{align*}
Since $K$ and $\varepsilon$ were arbitrary, by Proposition \ref{Prop_Mdwa_ucc}, $G\times H$ is $M_d$-weakly amenable with
\begin{align*}
\boldsymbol\Lambda(G\times H,d)< C_1C_2,
\end{align*}
which shows that
\begin{align*}
\boldsymbol\Lambda(G\times H,d)\leq\boldsymbol\Lambda(G,d)\boldsymbol\Lambda(H,d).
\end{align*}
Conversely, if we assume that $G\times H$ is $M_d$-weakly amenable, by \cite[Corollary 0.6]{Bat2}, both $G$ and $H$ are $M_d$-weakly amenable too.
\end{proof}

\subsection{Amenable groups}
We will also use the fact that amenable groups are $M_d$-weakly amenable. This result has already appeared in \cite[Corollary 2.6]{Ver} for discrete groups and in \cite[Remark 3.6]{Bat} for $\Z$, where it is mentioned that a similar proof works for any locally compact group. For completeness, we include here the proof of the general case. Let $G$ be a locally compact group, endowed with a left Haar measure $\mu$. Recall that $G$ is amenable if, for every compact subset $K\subset G$ and every $\varepsilon>0$, there is a measurable subset $U\subseteq G$ with $0<\mu(U)<\infty$ such that, for every $s\in K$,
\begin{align*}
\frac{\mu(sU\Delta U)}{\mu(U)}<\varepsilon.
\end{align*}
Moreover, in this case, the set $U$ may be assumed to be compact; see \cite[Theorem 7.3]{Pie} and \cite[Proposition 7.4]{Pie}.

\begin{lem}\label{Lem_am->Md-wa}
Let $G$ be a locally compact group. If $G$ is amenable, then it is $M_d$-weakly amenable with $\boldsymbol\Lambda(G,d)=1$ for every $d\geq 2$.
\end{lem}
\begin{proof}
Let us fix an integer $d\geq 2$, a compact subset $K\subseteq G$, and $\varepsilon>0$. Since $G$ is amenable, there is a compact, measurable subset $U\subseteq G$ with $0<\mu(U)<\infty$ such that, for all $s\in K$,
\begin{align*}
\frac{\mu(sU\Delta U)}{\mu(U)}<\varepsilon.
\end{align*}
Let $\lambda:G\to\mathbf{U}(L^2(G,\mu))$ be the left regular representation:
\begin{align*}
\lambda(s)f(t)=f(s^{-1}t).
\end{align*}
Let
\begin{align*}
\xi=\frac{1}{\mu(U)^{1/2}}\mathds{1}_U,
\end{align*}
where $\mathds{1}_U$ denotes the indicator function of the set $U$. Observe that $\xi$ is a unit vector in $L^2(G,\mu)$, and define, for every $s\in G$,
\begin{align*}
\varphi(s)=\langle\lambda(s)\xi,\xi\rangle=\frac{\mu(sU\cap U)}{\mu(U)}.
\end{align*}
Since $\lambda$ is a unitary representation, $\varphi$ is an element of $M_d(G)$ of norm at most $1$; see Lemma \ref{Lem_Btheta_Md}. Moreover, since $U$ is compact, $\varphi$ also belongs to $C_c(G)$. Furthermore, for every $s\in K$,
\begin{align*}
|1-\varphi(s)| & = \frac{\mu(U)-\mu(sU\cap U)}{\mu(U)}\\
& \leq \frac{\mu(sU\cup U)-\mu(sU\cap U)}{\mu(U)}\\
& = \frac{\mu(sU\Delta U)}{\mu(U)}\\
& < \varepsilon.
\end{align*}
By Proposition \ref{Prop_Mdwa_ucc}, we conclude that $G$ is $M_d$-weakly amenable with $\boldsymbol\Lambda(G,d)=1$.
\end{proof}

\subsection{Quotients}
We will also need the fact that the constants $\boldsymbol\Lambda(G,d)$ are stable under taking quotients by a compact subgroup.

\begin{lem}\label{Lem_Mdwa_G/K}
	Let $G$ be a locally compact group, $K$ a compact, normal subgroup of $G$, and $d\geq 2$. Then $G$ is $M_d$-weakly amenable if and only if $G/K$ is $M_d$-weakly amenable. Moreover,
	\begin{align*}
		\boldsymbol\Lambda(G,d)=\boldsymbol\Lambda(G/K,d).
	\end{align*}
\end{lem}
\begin{proof}
	Let $q:G\to G/K$ denote the quotient map. If $(\varphi_i)$ is an approximate identity in $M_d(G/K)$, then $(\varphi_i\circ q)$ is an approximate identity in $M_d(G)$ with
	\begin{align*}
		\|\varphi_i\circ q\|_{M_d(G)} \leq 	\|\varphi_i\|_{M_d(G/K)}.
	\end{align*}
	Moreover, if $\varphi_i$ is compactly supported, so is $\varphi_i\circ q$ because $K$ is compact. This shows that $\boldsymbol\Lambda(G,d)\leq\boldsymbol\Lambda(G/K,d)$. Now let $(\psi_i)$ be an approximate identity in $M_d(G)$, and define
	\begin{align*}
		\tilde{\psi}_i(s)=\int_K\psi_i(sk)\,dk
	\end{align*}
	for all $s\in G$, where $dk$ stands for the integration with respect to the normalised Haar measure on $K$. Using the fact that $G$ acts isometrically on $M_d(G)$ by right translations, one checks that
	\begin{align*}
		\|\tilde{\psi}_i\|_{M_d(G)} \leq 	\|\psi_i\|_{M_d(G)}.
	\end{align*}
	Moreover, if $\psi_i$ is compactly supported, so is $\tilde{\psi}_i$ because $K$ is compact. Finally, since $\tilde{\psi}_i$ is constant on each coset $sK$, it may be viewed as an element of $M_d(G/K)$. Again, by the compactness of $K$, $\tilde{\psi}_i$ is compactly supported on $G/K$ if it is compactly supported on $G$. This shows that $\boldsymbol\Lambda(G/K,d)\leq\boldsymbol\Lambda(G,d)$.
\end{proof}

\section{Baumslag--Solitar groups}\label{S_BS_gr}

In this section, we focus un Baumslag--Solitar groups and the proof of Theorem \ref{Thm_BS}, which relies on a construction of analytic families of uniformly bounded representations from \cite{Szw}. Let $\Omega$ be an open subset of $\C$, $G$ a group, and $\mathcal{H}$ a Hilbert space. For each $z\in\Omega$, let $\pi_z:G\to\mathbf{B}(\mathcal{H})$ be a representation. We say that the family $(\pi_z)_{z\in\Omega}$ is analytic if the map
\begin{align*}
	z\in\Omega\ \longmapsto\ \pi_z(t)\in\mathbf{B}(\mathcal{H})
\end{align*}
is holomorphic for each $t\in G$; see \cite[\S 3.3]{deCHaa} for different characterisations of Banach space valued holomorphic functions.

The following result is essentially an adaptation of \cite{Val} to our setting. It had already appeared in \cite[Proposition 3.2]{Ver} in the context of discrete groups, but here we will need to extend it to locally compact groups. We let $\mathbb{D}$ denote the open unit disk in $\C$. Recall that a function $\phi:G\to\N$ is proper if $\phi^{-1}(\{n\})$ is relatively compact for each $n\in\N$.

\begin{prop}\label{Prop_an_ub_rep}
Let $G$ be a locally compact group endowed with a proper, continuous function $l:G\to\N$ satisfying $l(e)=0$, where $e$ is the identity element of $G$. Assume that there is an analytic family of uniformly bounded representations $(\pi_z)_{z\in\D}$ of $G$ on a Hilbert space $\mathcal{H}$ such that $\pi_r$ is unitary for $r\in(0,1)$, $z\mapsto|\pi_z|$ is bounded on compact subsets of $\D$, and there is $\xi\in\mathcal{H}$ satisfying
\begin{align*}
z^{l(s)}=\langle\pi_z(s)\xi,\xi\rangle
\end{align*}
for all $z\in\D$, $s\in G$. Then $G$ is $M_d$-weakly amenable with $\boldsymbol\Lambda(G,d)=1$ for all $d\geq 2$.
\end{prop}
\begin{proof}
Fix $d\geq 2$ and define $\psi_z:G\to\C$ by
\begin{align*}
\psi_z(s)=z^{l(s)}
\end{align*}
for all $z\in\D$, $s\in G$. Then $z\mapsto\psi_z$ defines a holomorphic map from $\D$ to $M_d(G)$; see \cite[Lemma 3.1]{Ver}. We consider the F\'ejer kernel $F_N:\mathrm{S}^1\to\R$, defined on the unit circle $\mathrm{S}^1\subset\C$ by
\begin{align*}
F_N(z)=\sum_{|n|\leq N}\left(1-\frac{|n|}{N+1}\right)z^n
\end{align*}
for all $N\in\N$, $z\in \mathrm{S}^1$. Then $F_N\geq 0$ and, for every continuous function $f\in C(\mathrm{S}^1)$,
\begin{align*}
\lim_{N\to\infty}\frac{1}{2\pi}\int_0^{2\pi}F_N\big(e^{\ii\theta}\big)f\big(e^{\ii\theta}\big)\, d\theta = f(1);
\end{align*}
see \cite[Example 1.2.18]{Gra} for details. We define, for every $r\in(0,1)$ and $N\in\N$,
\begin{align*}
\Phi_{N,r}=\frac{1}{2\pi}\int_0^{2\pi}F_N\big(e^{\ii\theta}\big)\psi_{re^{\ii\theta}}\, d\theta.
\end{align*}
Observe that $\Phi_{N,r}$ belongs to $M_d(G)$. Moreover, for all $r\in(0,1)$,
\begin{align*}
\big\|\Phi_{N,r}-\psi_r\big\|_{M_d(G)}
&= \frac{1}{2\pi}\left\|\int_0^{2\pi}F_N\big(e^{\ii\theta}\big)\big(\psi_{re^{\ii\theta}}-\psi_r\big)\, d\theta\right\|_{M_d(G)}\\
&\leq \frac{1}{2\pi}\int_0^{2\pi}F_N\big(e^{\ii\theta}\big)\big\|\psi_{re^{\ii\theta}}-\psi_r\big\|_{M_d(G)}\, d\theta\\
&\quad \xrightarrow[N\to\infty]{} 0.
\end{align*}
In particular,
\begin{align*}
\lim_{N\to\infty}\big\|\Phi_{N,r}-\psi_r\big\|_\infty=0,
\end{align*}
and therefore
\begin{align*}
\lim_{r\to 1}\lim_{N\to\infty}\Phi_{N,r} = 1
\end{align*}
uniformly on compact subsets of $G$ because $l$ is proper. On the other hand, for every $s\in G$,
\begin{align*}
\Phi_{N,r}(s) &=\frac{1}{2\pi}\sum_{|n|\leq N}\left(1-\frac{|n|}{N+1}\right)\int_0^{2\pi}e^{\ii\theta n}r^{l(s)}e^{\ii\theta l(s)}\, d\theta\\
&= \begin{cases}
\left(1-\frac{l(s)}{N+1}\right)r^{l(s)},& \text{if } l(s)\leq N,\\
0,& \text{otherwise.}
\end{cases}
\end{align*}
This shows that $\Phi_{N,r}$ belongs to $C_c(G)$ because $l$ is continuous and proper. By Proposition \ref{Prop_Mdwa_ucc}, $G$ is $M_d$-weakly amenable with $\boldsymbol\Lambda(G,d)=1$.
\end{proof}

We will now apply this result to the automorphism group of a tree. Let $T$ be a locally finite tree, and let $\operatorname{Aut}(T)$ denote its automorphism group. For each $g\in\operatorname{Aut}(T)$, and each finite subset of vertices $S$ of $T$, we define
\begin{align*}
	U(g,S)=\left\{h\in\operatorname{Aut}(T)\ \mid\ \forall x\in S,\ h(x)=g(x)\right\},
\end{align*}
and we endow $\operatorname{Aut}(T)$ with the topology generated by all the subsets $U(g,S)$. With this topology, $\operatorname{Aut}(T)$ becomes a (totally disconnected) locally compact group. Moreover, if $d$ denotes the distance on $T$, and $x$ is any vertex, the function
\begin{align*}
	g\in\operatorname{Aut}(T)\ \longmapsto\ d(g(x),x)
\end{align*}
is continuous and proper.

\begin{cor}\label{Cor_Aut(T)}
Let $T$ be a locally finite tree, and $G=\operatorname{Aut}(T)$. Then $G$ is $M_d$-weakly amenable with $\boldsymbol\Lambda(G,d)=1$ for all $d\geq 2$.
\end{cor}
\begin{proof}
Let us fix a vertex $x\in T$, and let $\delta_x$ denote the delta function on $x$, viewed as an element of $\ell^2(T)$. By \cite[Theorem 1]{Szw}, there is an analytic family of uniformly bounded representations $(\pi_z)_{z\in\D}$ of $G$ on $\ell^2(T)$ such that, for all $z\in\D$, $g\in G$,
\begin{align*}
\langle\pi_z(g)\delta_x,\delta_x\rangle=z^{d(g(x),x)},
\end{align*}
\begin{align*}
|\pi_z|\leq 2\frac{|1-z^2|}{1-|z|},
\end{align*}
and $\pi_r$ is unitary for $r\in(-1,1)$; see also \cite{Val2}. By Proposition \ref{Prop_an_ub_rep}, $G$ is $M_d$-weakly amenable with $\boldsymbol\Lambda(G,d)=1$ for all $d\geq 2$.
\end{proof}

With all this, we can prove Theorem \ref{Thm_BS}.

\begin{proof}[Proof of Theorem \ref{Thm_BS}]
Let $G=\operatorname{BS}(m,n)$ and $d\geq 2$. Let us consider the semidirect product $\Z\ltimes_{\frac{n}{m}}\R$, where the action of $\Z$ on $\R$ is given by multiplication by $\frac{n}{m}$. Let $T$ be the Bass--Serre tree of $G$, viewed as an HNN extension; see \cite[\S 4.4]{Yan} for details. Then $T$ is the $(|m|+|n|)$-regular tree; see \cite[Theorem 4.10]{Yan}. By \cite[Theorem 1]{GalJan}, $G$ can be realised as a closed subgroup of the locally compact group $\big(\Z\ltimes_{\frac{n}{m}}\R\big)\times\operatorname{Aut}(T)$. On the other hand, since $\Z\ltimes_{\frac{n}{m}}\R$ is amenable, by Lemma \ref{Lem_am->Md-wa}, we have $\boldsymbol\Lambda\big(\Z\ltimes_{\frac{n}{m}}\R,d\big)=1$. Moreover, by Corollary \ref{Cor_Aut(T)}, $\boldsymbol\Lambda(\operatorname{Aut}(T),d)=1$. Hence, by Lemma \ref{Lem_Mdwa_prod}, $\big(\Z\ltimes_{\frac{n}{m}}\R\big)\times\operatorname{Aut}(T)$ is $M_d$-weakly amenable with
\begin{align*}
\boldsymbol\Lambda\big(\big(\Z\ltimes_{\frac{n}{m}}\R\big)\times\operatorname{Aut}(T),d\big)
\leq \boldsymbol\Lambda\big(\Z\ltimes_{\frac{n}{m}}\R,d\big) \boldsymbol\Lambda(\operatorname{Aut}(T),d)=1.
\end{align*}
This shows that $\boldsymbol\Lambda\big(\big(\Z\ltimes_{\frac{n}{m}}\R\big)\times\operatorname{Aut}(T),d\big)=1$. Finally, by \cite[Corollary 0.6]{Bat2}, $G$ is $M_d$-weakly amenable with $\boldsymbol\Lambda(G,d)=1$ because it is a closed subgroup of $\big(\Z\ltimes_{\frac{n}{m}}\R\big)\times\operatorname{Aut}(T)$.
\end{proof}

\section{Simple Lie groups with finite centre}\label{S_Lie_gr}

This section is devoted to the proof of Theorem \ref{Thm_Lie_gr}. We first recall the notion of real rank for simple Lie groups; for more details, we refer the reader to \cite{Hel2,Kna2}. Let $G$ be a simple Lie group, and let $\mathfrak{g}$ denote its Lie algebra. The Cartan decomposition of $\mathfrak{g}$ is given by
\begin{align*}
	\mathfrak{g}=\mathfrak{k}+\mathfrak{p},
\end{align*}
where $\mathfrak{k}$ and $\mathfrak{p}$ are the eigenspaces for the Cartan involution $\theta:\mathfrak{g}\to\mathfrak{g}$, associated to the eigenvalues $1$ and $-1$ respectively; see \cite[\S VI.2]{Kna2} for details. The real rank of $G$ --denoted by $\operatorname{rank}_\R G$-- is defined as the dimension of a maximal abelian subspace of $\mathfrak{p}$. For simple Lie groups, weak amenability and the exact value of the Cowling--Haagerup constant are completely determined by their real rank and their local isomorphism class; see \cite[\S 5]{Ver3} and the references therein. If $\operatorname{rank}_\R G\geq 2$, then $G$ is not weakly amenable. In particular, $\boldsymbol\Lambda(G,d)=\infty$ for all $d\geq 2$. If $\operatorname{rank}_\R G=0$, then $G$ is compact and therefore amenable. By Lemma \ref{Lem_am->Md-wa}, $\boldsymbol\Lambda(G,d)=1$ for all $d\geq 2$. Hence, the only case that requires a deeper analysis is when $\operatorname{rank}_\R G=1$.

We say that two Lie groups $G,H$ are locally isomorphic if their Lie algebras are isomorphic. In this case, we write $G\approx H$. As a consequence of the classification of simple real Lie algebras (see e.g. \cite[Theorem 6.105]{Kna2}), every connected simple Lie group of real rank 1 is locally isomorphic to either $\operatorname{F}_{4,-20}$, $\operatorname{SO}(n,1)$, $\operatorname{SU}(n,1)$ or $\operatorname{Sp}(n,1)$ ($n\geq 2$). Let us recall now the definitions of these four families of groups. Let $\R,\C,\mathbb{H}$ denote the real numbers, complex numbers, and quaternions respectively. For $n\geq 2$, we define
\begin{align*}
\operatorname{SO}(n,1) &= \left\{g\in\operatorname{SL}(n+1,\R)\ \mid\ g^*I_{n,1}g=I_{n,1}\right\},\\
\operatorname{SU}(n,1) &= \left\{g\in\operatorname{SL}(n+1,\C)\ \mid\ g^*I_{n,1}g=I_{n,1}\right\},\\
\operatorname{Sp}(n,1) &= \left\{g\in\operatorname{GL}(n+1,\mathbb{H})\ \mid\ g^*I_{n,1}g=I_{n,1}\right\},
\end{align*} 
where $g^*$ denotes the (conjugate) transpose of $g$, and $I_{n,1}$ is the diagonal matrix all whose non-zero entries are $1$, except for the last one, which is $-1$. The exceptional group $\operatorname{F}_{4,-20}$ is defined in similar fashion as the automorphism group of the hyperbolic plane over the octonions; see \cite{Val3} for details. The following result was proved in \cite{CowHaa}.

\begin{thm}[Cowling--Haagerup]\label{Thm_CH}
Let $G$ be a connected simple Lie group with finite centre and real rank $1$. Then $G$ is weakly amenable with
\begin{align*}
	\boldsymbol\Lambda(G)=\begin{cases}
		1 & \text{if }\ G\approx\operatorname{SO}(n,1),\ n\geq 2,\\
		1 & \text{if }\ G\approx\operatorname{SU}(n,1),\ n\geq 2,\\
		2n-1 & \text{if }\ G\approx\operatorname{Sp}(n,1),\ n\geq 2,\\
		21 & \text{if }\ G\approx\operatorname{F}_{4,-20}.
	\end{cases}
\end{align*}
\end{thm}

We will show that the same characterisation holds for $M_d$-weak amenability, although we are not able to compute the exact values of the constants $\boldsymbol\Lambda(\operatorname{Sp}(n,1),d)$  and $\boldsymbol\Lambda(\operatorname{F}_{4,-20},d)$ for $d\geq 3$.

\begin{lem}\label{Lem_Mdwa_rk1}
	Let $G$ be either $\operatorname{F}_{4,-20}$, $\operatorname{SO}(n,1)$, $\operatorname{SU}(n,1)$ or $\operatorname{Sp}(n,1)$ ($n\geq 2$). For every $d\geq 2$, $G$ is $M_d$-weakly amenable. Moreover,
	\begin{align*}
		&\boldsymbol\Lambda(G,d)=1  & &\text{if }\ G=\operatorname{SO}(n,1) \text{ or } G=\operatorname{SU}(n,1),\\
		2n-1 \leq\ &\boldsymbol\Lambda(G,d) \leq (2n-1)^d  & &\text{if }\ G=\operatorname{Sp}(n,1),\\
		21 \leq\ &\boldsymbol\Lambda(G,d) \leq (21)^d & &\text{if }\ G=\operatorname{F}_{4,-20}.
	\end{align*}
\end{lem}
\begin{proof}
	For $d=2$, the result is a consequence of Theorem \ref{Thm_CH} since $M_2$-weak amenability is the same as weak amenability, and $\boldsymbol\Lambda(G,2)=\boldsymbol\Lambda(G)$. Now let $d\geq 3$ and $\theta>\boldsymbol\Lambda(G)$. It was shown in (the proof of) \cite[Theorem 1.5]{Ver2} that there is a sequence $(\varphi_n)$ in $C_c(G)$ such that
	\begin{align*}
		\limsup_{n\to\infty}\|\varphi_n\|_{B_\theta(G)}\leq 1,
	\end{align*}
	and
	\begin{align*}
		\lim_{n\to\infty}\varphi_n = 1 \qquad\text{in }\quad \sigma\big(B_\theta(G),\tilde{A}_\theta(G)\big).
	\end{align*}
	We should mention that the results in \cite{Ver2} are only stated for $\operatorname{Sp}(n,1)$ and $\operatorname{F}_{4,-20}$ because that article focuses on those groups, but they are also true for $\operatorname{SO}(n,1)$ and $\operatorname{SU}(n,1)$ since the proof depends only on the construction of representations given by \cite[Theorem 2.1]{Doo}, which is proved for all four classes of groups. By Lemma \ref{Lem_Btheta_Md}, we have
		\begin{align*}
		\limsup_{n\to\infty}\|\varphi_n\|_{M_d(G)}\leq \theta^d,
	\end{align*}
	and
	\begin{align*}
		\lim_{n\to\infty}\varphi_n = 1 \qquad\text{in }\quad \sigma\big(M_d(G),X_d(G)\big).
	\end{align*}
	We conclude that $G$ is $M_d$-weakly amenable and $\boldsymbol\Lambda(G,d)\leq\boldsymbol\Lambda(G)^d$. Since we always have $\boldsymbol\Lambda(G,d)\geq\boldsymbol\Lambda(G)$, the result follows from Theorem \ref{Thm_CH}.
\end{proof}

Now we are ready to prove Theorem \ref{Thm_Lie_gr}.

\begin{proof}[Proof of Theorem \ref{Thm_Lie_gr}]
	Let $G$ be a simple Lie group with finite centre. If $\operatorname{rank}_\R G=0$, then $G$ is compact, and therefore $\boldsymbol\Lambda(G,d)=1$ by Lemma \ref{Lem_am->Md-wa}. If $\operatorname{rank}_\R G\geq 2$, then $\boldsymbol\Lambda(G,2)=\infty$ by \cite[Theorem 1]{Haa}. Therefore $\boldsymbol\Lambda(G,d)=\infty$ for all $d\geq 3$. Now assume that $\operatorname{rank}_\R G=1$. As discussed above, $G$ is locally isomorphic to $H$, where $H$ is either $\operatorname{F}_{4,-20}$, $\operatorname{SO}(n,1)$, $\operatorname{SU}(n,1)$ or $\operatorname{Sp}(n,1)$ ($n\geq 2$). Let $Z(G)$ denote the centre of $G$. By \cite[Corollary II.5.2]{Hel2}, $G/Z(G)$ is isomorphic to $H/Z(H)$. Therefore, by Lemma \ref{Lem_Mdwa_G/K},
	\begin{align*}
		\boldsymbol\Lambda(G,d)=\boldsymbol\Lambda(G/Z(G),d)=\boldsymbol\Lambda(H/Z(H),d)=\boldsymbol\Lambda(H,d).
	\end{align*}
	The result then follows from Lemma \ref{Lem_Mdwa_rk1}.
\end{proof}


\begin{ack}
I am grateful to Mikael de la Salle for inspiring discussions. I also thank Michael Cowling for kindly answering my questions about the representation theory of Lie groups, which I was hoping to use in order to obtain an improved version of Theorem \ref{Thm_Lie_gr}; see Remark \ref{Rmk_G=KS}.
\end{ack}

\begin{funding}
This work is supported by the FONDECYT project 3230024 and the ECOS--ANID project 23003 \emph{Small spaces under action}.
\end{funding}


\bibliographystyle{plain} 

\bibliography{Bibliography}

\end{document}